\documentclass[12pt, reqno]{amsart}
\usepackage{amsmath, amsthm, amscd, amsfonts, amssymb, graphicx, color}
\usepackage[bookmarksnumbered, colorlinks, plainpages]{hyperref}
\hypersetup{colorlinks=true,linkcolor=red, anchorcolor=green, citecolor=cyan, urlcolor=red, filecolor=magenta, pdftoolbar=true}

\usepackage{tensor}

\newtheorem{theorem}{Theorem}[section]

\newtheorem{proposition}[theorem]{Proposition}
\newtheorem{corollary}[theorem]{Corollary}
\theoremstyle{definition}
\newtheorem{definition}[theorem]{Definition}
\newtheorem{example}[theorem]{Example}

\newtheorem{remark}[theorem]{Remark}

\numberwithin{equation}{section}

\usepackage{fullpage}

\begin{document}
	
	\setcounter{page}{1}
	
	\title[PMS-completeness]{A new characterization of partial metric completeness}

	\author[Ya\'e Ulrich Gaba]{Ya\'e Ulrich Gaba$^{1,2,\dagger}$}

	\address{$^{1}$ Institut de Math\'ematiques et de Sciences Physiques (IMSP), 01 BP 613 Porto-Novo, B\'enin.}

	\address{$^{2}$ African Center for Advanced Studies (ACAS),
		P.O. Box 4477, Yaounde, Cameroon.}

	\email{\textcolor[rgb]{0.00,0.00,0.84}{yaeulrich.gaba@gmail.com
	}}

	\subjclass[2010]{Primary 47H05; Secondary 47H09, 47H10.}
	
	\keywords{metric type space, fixed point, $\lambda$-sequence.}

		\subjclass[2010]{Primary 47H05; Secondary 47H09, 47H10.}

	\keywords{Partial metric; completeness; fxed point.}
	
	\date{Received: xxxxxx; Accepted: zzzzzz.
		\newline \indent $^{\dagger}$Corresponding author}
	
	\begin{abstract}
		In this article, we present a new characterization of the completeness of a partial metric space--which we call \textit{orbital characterization}-- using fixed point results.

		\end{abstract} 
	
	\maketitle
	
	\section{Introduction and preliminaries}
	
It is known that the Banach contractive mapping theorem does not characterize the complete metric spaces. The first characterization of completeness in terms of contraction was done by Hu \cite{hu67}.
\begin{theorem}\label{t.Hu}
	A metric space $(X,\rho)$ is complete if and only if  for every nonempty closed subset $Y$ of $X$ every contraction on $Y$ has a fixed point in $Y$.
\end{theorem}

Subrahmanyam \cite{sub} proved the following completeness result.
\begin{theorem}\label{t.Subram}
	A metric space $(X,\rho)$ in which every mapping $f:X\to X$ satisfying the conditions
	
	{\rm (i)}\; there exists $\alpha>0$ such that  $$\rho(f(x),f(y))\le \alpha \max\{\rho(x,f(x)),\rho(y,f(y))\} \text{ for all } x,y\in X;$$ 
	
	{\rm (ii)}\; $f(X)$ is countable;\\
	has a fixed point, is complete.
\end{theorem}

The condition (i) in this theorem is related to Kannan  and Chatterjea conditions: there exists $\alpha\in(0,\frac12)$ such that  for all $x,y\in X,$
\begin{equation}\tag{K}\label{Kan}
\rho(f(x),f(y))\le \alpha \,\left[\rho(x,f(x))+\rho(y,f(y))\right]\,,
\end{equation}
respectively
\begin{equation}\tag{Ch}\label{Chat}
\rho(f(x),f(y))\le \alpha  \,\left[\rho(x,f(y))+\rho(y,f(x))\right]\,.
\end{equation}

Kannan and Chatterjea  proved that any mapping $f$ on a complete metric space satisfying  \eqref{Kan} or \eqref{Chat} has a fixed point.  As it is remarked in \cite{sub}, Theorem \ref{t.Subram} provides completeness of metric spaces on which every Kannan, or every Chatterjea map, has a fixed point.

In \cite{kirk}, Kirk also proposed a solution of this problem using Caristi mappings, since he proved that a metric space $(X, d)$ is complete
if, and only if, every Caristi mapping for $(X, d)$ has a fixed point. More recently several authors have obtained theorems of
fixed point for generalized metric spaces and some of these generalizations are in the setting of the partial
metrics (see \cite{alt,rom} for e.g.). 

Partial metrics were introduced by Matthews\cite{mat} in 1992. They generalize the concept of a metric space in the sense
that the self-distance from a point to itself need not be equal to zero.

Ge and Lin \cite{ge} recently introduced the notions of symmetrical density and
sequential density in partial metric spaces and proved the existence and uniqueness
theorems in the classical sense for the completion of partial metric spaces. The main results of Ge and Lin are as follows.

\begin{theorem}\cite[Theorem 1]{ge})
	Every partial metric space has a completion.
\end{theorem}

\begin{theorem} \cite[Proposition 1]{ge}) Let $(X ^* , p)$ and $(Y^* , q)$ be two complete partial
metric spaces, $X$ and $Y$ be symmetrically dense subsets of $X^*$ and $Y^*$ respectively.
If $h : X \to Y$ is an isometry then there is a unique isometric extension $f : X^*\to
Y^*$ which is an extension of $h$.
\end{theorem}

 Our aim is to give a characterization theorem for the completeness of partial metric spaces via a certain fixed point theory and using some contractive conditions on orbits.

\begin{definition} (Compare \cite{mat})
	A partial metric type on a set $X$ is a function $p: X \times X \to [0, \infty)$ such that:
	\begin{enumerate}
		\item[(pm1)] $x = y$ iff $(p(x, x) = p(x, y)=p(y,y)$ whenever $x, y \in X$,
		\item[(pm2)] $0\leq p(x, x)\leq p(x, y)$ whenever $x, y \in X$,
		
		\item[(pm3)] $p(x, y) = p(y, x);$ whenever $x, y \in X$,
		
		\item[(pm4)] $$p(x, y) + p(z,z)\leq  p(x,z)+p(z,y)$$ 
		for any points $x,y,z\in X$. 
		
	\end{enumerate}

	The pair $(X, p)$ is called a partial metric space.

\end{definition}
It is clear that, if $p ( x , y ) = 0$ , then, from (pm1) and (pm2), $x = y$.

The family $\mathcal B'$ of sets
\begin{equation}\label{def.basis2-pm}
B'_p(x,\varepsilon):=\{y\in X : p(x,y)<\varepsilon +p(x,x)\},\; x\in X,\, \varepsilon>0\,,\end{equation}
is a basis for a topology $\tau(p)$ on $X$. The topology $\tau(p)$ is $T_0$.

\begin{definition}
	Let $(X, p)$ be a partial metric space. Let $(x_n)_{n\geq 1}$ be any sequence in $X$ and $x \in X$. Then:
	
	\begin{enumerate}
		\item The sequence $(x_n)_{n\geq 1}$ is said to be convergent with respect to $\tau(p)$ (or $\tau(p)$-convergent) and converges to $x$, if $\lim\limits_{n\to \infty} p(x , x_n) = p(x, x)$.
		We write $$x_n \overset{p}{\longrightarrow} x.$$

		\item The sequence $(x_n)_{n\geq 1}$ is said to be a $p$-Cauchy sequence if
		$$\lim\limits_{n\to \infty,m\to \infty} p(x_n , x_m)$$ exists and is finite.

	\end{enumerate}

	\ $(X, p)$ is said to be complete if for every $p$-Cauchy sequence $(x_n)_{n\geq 1} \subseteq X$, there exists $x \in X$ such that:
		
		$$ \lim\limits_{n\to \infty,m\to \infty} p(x_n , x_m)= \lim\limits_{n\to \infty} p(x , x_n)=p(x,x).$$
	
\end{definition}

\begin{remark}
The reader might have observed that this definition of completeness is stronger than the one in metric spaces. Since the topology $\tau(p)$ of a partial metric space is only $T_0$, a convergent sequence can have many limits (see \cite[Example 2]{shu}). In fact, if $x_n\xrightarrow{p}x$, then $x_n\xrightarrow{p}y$ for any $y\in X$ such that  $p(x,y)=p(y,y)$.
 Indeed
$$
0\le  p(y,x_n)-p(y,y)\le p(y,x)+p(x,x_n)-p(x,x)-p(y,y)=p(x,x_n)-p(x,x)\longrightarrow 0\,.$$

To obtain uniqueness and to define a reasonable notion of completeness, a stronger notion of convergence is needed.
\end{remark}

\begin{definition}\label{def.prop-converg-pm-sp}
One says that a sequence $(x_n)$ in a partial metric space \emph{converges properly} to $x\in X$ iff
\begin{equation}\label{eq1.prop-converg-pm-sp}
\lim_{n\to\infty}p(x,x_n)=p(x,x)=\lim_{n\to\infty} p(x_n,x_n)\,.
\end{equation}

We shall write $x_n \overset{ppr}{\longrightarrow} x.$
\end{definition}

In other words, $(x_n)$  converges properly to $x$ if and only if    $(x_n)$   converges to $x$ with respect to  $\tau(p)$ and further
\begin{equation}\label{eq2.prop-converg-pm-sp}
\lim_{n\to\infty}p(x_n,x_n)=p(x,x)\,.\end{equation}

\begin{proposition}\label{p3.pm-sp}
Let $(X,p)$ be a partial metric space and  $(x_n)$  a sequence in $X$ that  converges properly to $x\in X$. Then
\begin{enumerate}
	\item[{\rm (i)}]\;\; the limit is unique, and
	\item[{\rm (ii)}]\;\; $\lim\limits_{m,n\to\infty}p(x_m,x_n)=p(x,x).$\end{enumerate}
\end{proposition}\begin{proof}
Suppose that $x,y\in X$ are such that  $(x_n)$ converges properly  to both $x$ and $y$.
Then
$$
p(x,y)\le p(x,x_n)+p(x_n,y)-p(x_n,x_n) \longrightarrow p(y,y)   \quad\mbox{as}\; n\to \infty\,,
$$
implying $p(x,y)\le p(y,y)$. But, by (pm2), $ p(y,y)\le p(x,y)$, so that
\begin{equation}\label{eq.prop-converg-pm}
p(x,y)= p(y,y)=p(x,x)\,,\end{equation}
which by (pm1) yields $x=y$.

To prove (ii) observe that
$$ p(x_m,x_n)\le p(x_m,x)+p(x,x_n)-p(x,x)$$
so that
$$ p(x_m,x_n)-p(x,x)\le p(x_m,x)-p(x,x_n)-2p(x,x)\longrightarrow 0\;\;\mbox{as}\; m,n\to \infty\,.$$

Also
\begin{align*}
p(x,x)\le& p(x,x_m)+p(x_m,x)-p(x_m,x_m)\\
\le& p(x,x_m)+p(x_m,x_n)+p(x_n,x)-p(x_n,x_n)-p(x_m,x_m)\,,
\end{align*}
implies
$$
p(x,x)-p(x_m,x_n)\le p(x,x_m)+p(x_n,x)-p(x_n,x_n)-p(x_m,x_m)\longrightarrow 0\;\;\mbox{as}\; m,n\to \infty\,.$$

Consequently $\lim\limits_{m,n\to\infty}p(x_m,x_n)=p(x,x)$.
\end{proof}
\begin{remark}\label{re.converg-seq-pm-sp}
Some authors take the condition (ii) from Proposition \ref{p3.pm-sp} in the definition of a properly
convergent sequence. As it was shown this  is equivalent to the condition from Definition \ref{def.prop-converg-pm-sp}
\end{remark}

The definition of $p$-Cauchy sequences in partial metric spaces takes the following equivalent form.
\begin{proposition}\label{def.C-seq-pm-sp}
	A sequence $(x_n)$ in a partial metric space $(X,p)$ is called a \emph{Cauchy sequence} if there exists $a\ge 0$ in $\mathbb{R}$ such that  for every $\varepsilon>0$ there exists $n_\varepsilon\in\mathbb{N}$ with
	$$
	|p(x_n,x_m)-a|<\varepsilon\,,
	$$
	for all $m,n\ge n_\varepsilon$, written also as $\lim\limits_{m,n\to \infty}p(x_n,x_m)=a$.

\end{proposition}

\begin{definition}
	The partial metric space $(X,p)$ is called \emph{complete} if every $p$-Cauchy sequence is properly convergent to some $x\in X$.
\end{definition}

Let $f$ be a self mapping of a partial metric space $(X,p)$. For $x,y\in X$ and $A\subset X$, define 

\[O(x,f):= \{f^nx, n\in \mathbb{N}\}, \ \  O(x,y,f)=O(x,f) \cup O(y,f),\] 
 
\[\delta(A):=\sup \{p(x,y);\ x,y,\in A\} . \]

We conclude this introductory part by introducing this family of functions.
Define a family $\Psi$ of functions as follows: $\psi \in \Psi \iff$

$ \psi :[0,\infty) \to [0,\infty) \text{ is nondecreasing, continuous on the right and} \psi(t)<t \text{ for all } t>0.$

From the definition of $\Psi$, it is easy to see that if $(x_n)_{n\geq 1} \subset [0,\infty)$ is a sequence that satisfies $x_{n+1} \leq \psi(x_n)$ for some $\psi \in \Psi $, then $x_n\to 0$ as $n\to \infty.$

\section{Main result}

Our characterization is as follows:

\begin{theorem}\label{orbital}(Orbital characterisation)
	For a partial metric space $(X,p)$, the following statements are equivalent:
	\begin{enumerate}
	\item[(1)] $(X,p)$ is complete;
	
	\item[(2)] If $f$ is a self mapping of $X$ satisfying for every $x,y\in X$ and some $\psi \in \Psi $,
	
	\begin{equation}\tag{a}\label{Car1}
	p(fx,fy) \leq \psi(\delta(O(x,y,f))), \qquad \delta(O(x,y,f))<\infty 
	\end{equation}
then $f$ has a fixed point;	 
	
	\item[(3)] If $f$ is a self mapping of $X$ satisfying for all $x,y\in X$ and some $r\in [0,1) $,
	
	\begin{equation}\tag{b}\label{Car2}
	p(fx,fy) \leq r\delta(O(x,y,f)), \qquad \delta(O(x,y,f))<\infty 
	\end{equation}
then $f$ has a fixed point.	
	\end{enumerate}
	
\end{theorem}

\begin{proof}
$(1) \Longrightarrow (2)$ For any $x_0,y_0 \in X$ and $n\in \mathbb{N}$, define the sequences $(x_n)$ and $(y_n)$ by $x_n = f^nx_0$ and $y_n = f^ny_0$. Using the hypothesis from (2), we have, for any $k,m\geq n$

\[ p(fx_k,fy_m) \leq \psi(\delta(O(x_k,y_m,f))) . \]

In particular, $$ p(x_{n+1},y_{n+1}) \leq \psi(\delta(O(x_n,y_n,f))).$$ 
If follows from the definition of $\Psi$, that $\delta(O(x_n,y_n,f)) \to 0$ as $n\to \infty.$ Therefore $(x_n)$ and $(y_n)$ are $p$-Cauchy sequences. By completeness of $X$, there exists $x^*\in X$ such that $x_n \overset{ppr}{\longrightarrow} x^*,$ i.e.

$$ \lim\limits_{n\to \infty,m\to \infty} p(x_n , x_m)= \lim\limits_{n\to \infty} p(x^* , x_n)=p(x^*,x^*).$$

Note that \begin{align*} p(y_n,x^*) &\leq p(y_n,x_n)+ p(x_n,x^*)-p(x_n,x_n)\\
&\leq \delta(O(x_n,y_n,f)) + p(x_n,x^*)-p(x_n,x_n)\\
& \leq \delta(O(x_n,y_n,f)) + p(x_n,x^*)
\end{align*}

Taking the limit as $n\to \infty$, we have 

\[ \lim\limits_{n\to \infty} p(y_n,x^*)  \leq \lim\limits_{n\to \infty}\delta(O(x_n,y_n,f)) + \lim\limits_{n\to \infty} p(x_n,x^*)=p(x^*,x^*) , \]

i.e $$\lim\limits_{n\to \infty} p(y_n,x^*)  \leq p(x^*,x^*) .$$

Moreover since 

\begin{align*} p(x_n,x^*) &\leq p(x_n,y_n)+ p(y_n,x^*)-p(y_n,y_n)\\
&\leq \delta(O(x_n,y_n,f)) + p(y_n,x^*),\\
\end{align*}

we obtain

\[p(x^*,x^*)\leq \lim\limits_{n\to \infty} p(y_n,x^*).  \]

Therefore $$\lim\limits_{n\to \infty} p(y_n,x^*) = p(x^*,x^*).$$

From $p(y_n,x^*) \leq p(y_n,x_n)+ p(x_n,x^*)-p(x_n,x_n)$ and taking the limit as $n\to \infty$, we have $$p(x^*,x^*)=0.$$

On the other hand, observe that we have 
\[ p(y_n,y_m) \leq p(y_n,x^*)+p(x^*,y_n)-p(x^*,x^*)    \]

and

\[p(x^*,x^*)\leq p(x^*,y_n) + p(y_n,y_m)+p(y_m,x^*)    \]

respectively; and taking the limit as $n,m\to \infty$ in both inequalities, we obtain 

$$\lim\limits_{n\to \infty,m\to \infty} p(y_n , y_m)= \lim\limits_{n\to \infty} p(x^* , y_n)=p(x^*,x^*)=0,$$

i.e. $$y_n \overset{ppr}{\longrightarrow} x^*.$$

We shall now prove that $x^*$ is a fixed point for $f$ and for that, it is enough to establish that $\delta(O(x^*,f))=0.$

By the way of contradiction, assume that $\delta(O(x^*,f))>0.$ Then for any $n,n \in \mathbb{N},$ we have

\[ p(f^nx^*,f^mx^*) \leq    \psi(\delta(O(f^{n-1}x^*,f^{m-1}x^*,f))) \leq \psi(\delta(O(x^*,f))),                  \]

which means in particular that

\[\delta(O(fx^*,f)) \leq \psi(\delta(O(x^*,f)))< \delta(O(x^*,f)). \]

It follows that 

\begin{equation} \tag{c}\label{Car3}
\delta(O(x^*,f)) =\sup  \{p(x^*,f^mx^*),m\in \mathbb{N}  \} ,  \end{equation}

since $$\delta(O(x^*,f)) = \max \{\sup  \{p(x^*,f^mx^*),m\in \mathbb{N}  \},\delta(O(fx^*,f)) \}.$$
And this value can't be $\delta(O(fx^*,f))$ since $\delta(O(fx^*,f)) < \delta(O(x^*,f)). $
\end{proof}

In view of $\lim\limits_{n\to \infty} p(x^* , x_n)=p(x^*,x^*)=0$, for every $\varepsilon>0$, there exists $k\in \mathbb{N}$ such that $p(x_n,x^*)<\varepsilon$ whenever $n\geq k.$

Consequently, whenever $n\geq k$ and for any $m\in \mathbb{N}$ 

\begin{align*}
 p(x^*,f^mx^*) &\leq p(x^*,x_n) + p(x_n,x_m) - p(x_n,x_n)\\  
 &\leq \varepsilon + \psi(\delta(O(f^{n-1}x^*,f^{m-1}x^*,f))) \\
 &\leq \varepsilon + \psi(\max\{2\varepsilon,\delta(O(x^*,f))+\varepsilon \} ).
  \end{align*}

Hence \[\sup \{p(x^*,f^mx^*), m\in \mathbb{N}\} \leq \varepsilon + \psi(\max\{2\varepsilon,\delta(O(x^*,f))+\varepsilon \} ). \]

Letting $\varepsilon \to 0$, by the above inequality and \eqref{Car3}, we have 

\[ \delta(O(x^*,f)) =\sup  \{p(x^*,f^mx^*),m\in \mathbb{N}  \}  \leq  \psi(\delta(O(x^*,f))) <   \delta(O(x^*,f)),  \]

--a contradiction. Hence $\delta(O(x^*,f))=0$ and $fx^* = x^*.$

\vspace*{0.3cm}

$(2) \Longrightarrow (3)$ Take $\psi(t) =rt$ for all $t\in [0,\infty).$

\vspace*{0.3cm}

$(3) \Longrightarrow (1)$ Suppose that $(X,p)$ is not complete. Let $X^*$ be an isometric completion of $X$ (in the sense of Ge \cite{ge}). Then there exists a $p$-Cauchy sequence $(x_n) \subset X$ and a point $u\in X^*\setminus X$ such that $x_n \overset{ppr}{\longrightarrow} u.$ Take $b = \frac{1}{5}$ and $r=\frac{1}{2}$ and define $P_n = \{x\in X:p(x,u)\leq b^n\}$ where $n\in \mathbb{Z}$. One easily sees that $P_n$ is nonempty for each $n\in \mathbb{Z}$ and $$X =\underset{n\in \mathbb{Z}}{\bigcup}  P_n.$$

Set $n(x) = \max \{n: x\in P_n  \}$ for any $x\in X$. Since $x_i \overset{ppr}{\longrightarrow} x$, for each $n\in \mathbb{N}$, there exists a smallest $k(n)$ such that $x_i \in P_n$ for $i\geq k(n).$ Define a self mapping $f$ on $X$ by

\[f(x) = \begin{cases} x_{k(2)} &\mbox{if } n(x)\leq 0 \\
 x_{k(n(x)+2)}	& \mbox{if } n (x) >0 \end{cases}
\]

for each $x\in X.$ Obviously, $f$ has no fixed point. Observe that $f(X)\subset P_1$ and hence $f$ is bounded\footnote{Meaning that there exists $M>0$ such that $p(fx,fy)\leq M$ whenever $x,y\in X$.}. Moreover, 

$$p(x,f^nx) \leq p(x,fx)+ p(fx,f^nx) -p(fx,fx)\leq p(x,fx) + \delta(f(X)),$$

for each $x\in X$ and each $n\in \mathbb{N}$. It follows that

\[ \delta(O(x,y,f)) \leq p(x,y) + p(x,fx)+p(y,fy) + \delta(f(X)) <\infty  \]

for all $x,y\in X.$ It is easy to verify that $fx \in P_{n(x)+2}$ for each $x\in X$, i.e.

\[ p(fx,u)\leq b^{n(x)+2}. \]

By definition of $n(x), $ we have that $b^{n(x)+1}\leq p(x,u),$ hence

\[ b ^{n(x)+2} \leq bp(x,u), \]

which means that

\[ p(fx,u) \leq bp(x,u) \leq b [p(x,fx)+p(fx,u)]. \]

It follows that for each $x\in X$

\[ p(fx,u) \leq \frac{b }{1-b}p(x,fx). \]

Consequently, for each $x,y \in X$

\begin{align*}
p(fx,fy) \leq p(fx,u) + p(u,fy) &\leq \frac{b }{1-b} [p(x,fx)+ p(y,fy)] \\
& \leq r \max\{ p(x,fx),p(y,fy)\}\\
& \leq r \delta(O(x,y,f)),
\end{align*}

that is, $f$ satisfies \eqref{Car2}. By (3), $f$ has a fixed point and this is a contradiction. The proof is complete.

A direct consequence of the above theorem is the following:

\begin{corollary}\label{theor2.2}
	If $T$ is a self mapping of a complete partial metric space $(X,p)$ satisfying \eqref{Car1} for all $x,y\in X$ and some $\psi \in \Psi$, then $f$ has a unique fixed point $x^*$ and $T^nx \overset{ppr}{\longrightarrow} x^*$ for each $x\in X.$
\end{corollary}

\begin{proof}
	It follows from Theorem \ref{orbital} that for each $x\in X$, there exists a fixed point $x^*$ of $f$ such that $T^nx \overset{ppr}{\longrightarrow} x^*$. Condition \eqref{Car1} ensures that $T$ has a unique fixed point.
\end{proof}

We conclude this manuscript with the following example:

\begin{example}
	Let $X=[0,\infty)$ with the partial metric $p(x,y)=\max\{x,y\}$ and take $\psi(t)=rt$ for $t\in [0,\infty)$, where $r\in [0,1)$. Define the self mapping $T$ on $X$ by $Tx=rx$. It is easily seen that, for all $x,y \in X,$
	
	\[p(Tx,Ty)\leq r p(x,y)\leq r\delta(O(x,y,T)) \]
	
	and 
	
	\[ \delta(O(x,y,T)) \leq x+y <\infty. \]
	
	Hence the conditions of Theorem \ref{theor2.2} are satisfied and the fixed point is $x=0.$
	
	Moreover, for any $x\in X,$ $x_n:=T^nx=r^nx$ for $n\geq 1.$ Hence
	
	\[ \lim\limits_{n<m,n,m\to \infty}p(x_n,x_m) = \lim\limits_{n\to \infty}r^nx=0 = p(0,0)=\lim\limits_{n\to \infty}p(x_n,0),  \]
	
	i.e. 
	$$T^nx \overset{ppr}{\longrightarrow} 0.$$
	
\end{example}

	\bibliographystyle{amsplain}

\end{document}